\documentclass[a4paper,12pt]{article}

\usepackage{amsmath}
\usepackage{amssymb}
\usepackage{amsthm}
\usepackage{enumerate}
\usepackage{graphicx}
\usepackage[hmargin=2.5cm,vmargin=2.5cm]{geometry}
\usepackage{parskip}
\usepackage{color}

\newtheorem{definition}{Definition}[section]

\newtheorem{theorem}{Theorem}[section]

\newtheorem{lemma}{Lemma}[section]
\newtheorem{proposition}{Proposition}[section]

\DeclareMathOperator{\prob}{\mathbb{P}}
\DeclareMathOperator{\cov}{Cov}

\author{Alef E.~Sterk\footnote{Bernoulli Institute for Mathematics, Computer Science and Artificial Intelligence,
University of Groningen, PO Box 407, 9700 AK Groningen, The Netherlands. E-mail: a.e.sterk@rug.nl}}
\title{Max-semistable extreme value laws for autoregressive processes with Cantor-like marginals}
\date{\today}

\begin{document}

\maketitle

\begin{abstract}
This paper considers a family of autoregressive processes with marginal distributions resembling the Cantor function. It is shown that the marginal distribution is in the domain of attraction of a max-semistable distribution. The main result is that the extreme value law for the  autoregressive process is obtained by including an extremal index in the law for an i.i.d.\ process with the same marginal distribution. Connections with extremes in deterministic dynamical systems and the relevance of max-semistable distributions in that context are also pointed out.
\end{abstract}




\newpage

\section{Introduction}

Let $X_0,X_1,X_2,\dots$ be a stationary sequence of random variables with marginal distribution function $F$. Extreme value theory is concerned with finding the possible limit laws for the partial maximum $M_n = \max\{X_0,\dots,X_{n-1}\}$ under suitable normalizations. In this paper we focus  on linear normalizations of the subsequence $M_{k_n}$ where $k_n$ has a geometric growth rate. That is, we consider the convergence of
\begin{equation}
\label{eq:convergence_max}
\lim_{n\to\infty} \prob(a_n(M_{k_n} - b_n) \leq x) = G(x)
\end{equation}
subject to the condition
\begin{equation}
\label{eq:Megyesi:1.2}
\lim_{n\to\infty} \frac{k_{n+1}}{k_n} = c \geq 1.
\end{equation}

If the variables $X_k$ are independent, then equation \eqref{eq:convergence_max} amounts to
\begin{equation}
\label{eq:Megyesi:1.1}
\lim_{n\to\infty} F(x/a_n + b_n)^{k_n} = G(x).
\end{equation}
For $c=1$ it can be shown that $F$ is necessarily in the domain of attraction of a \emph{max-stable distribution} \cite{Megyesi:02, Pancheva:92} and the possible forms of $G$ are well-known in the literature on extremes \cite{Beirlant:2004, Embrechts, Galambos:78, HF:2006, Resnick:1987}. The limit laws for stationary sequences of dependent variables are studied in detail in \cite{Leadbetter1980}.

A different class of distributions is obtained for $c>1$. In that case, $F$ is said to be in the domain of attraction of a \emph{max-semistable distribution} and (up to changes in scale and location) $G$ is of the form
\begin{equation}
\label{eq:maxsemistable_family}
G(x)
=
\begin{cases}
\exp\{ -(1+\xi x)^{-1/\xi} \nu(\log((1+\xi x)^{-1/\xi})) \} & \text{if } \xi \neq 0 \text{ and } 1+\xi x>0, \\
\exp\{ -e^{-x}\nu(x) \} & \text{if } \xi = 0 \text{ and } x \in \mathbb{R},
\end{cases}
\end{equation}
where $\xi \in \mathbb{R}$ and the function $\nu$ is positive, bounded, and periodic with period $\log(c)$ \cite{Grinevich:92, Grinevich:94, Pancheva:92}. An alternative representation for $G$ is derived in \cite{CantoHaanTemido:2002}. Limit laws for stationary sequences of dependent variables, including an extension of the concept of extremal index, are discussed in \cite{TemidoCanto:2003}.

In this paper we consider autoregressive processes of the form
\begin{equation}
\label{eq:AR_proc}
X_{k+1} = \beta X_{k} + \varepsilon_{k+1}, \quad k \geq 0,
\end{equation}
where $X_0$ is a random variable and $(\varepsilon_k)$ is an i.i.d.\ process with
\[
\prob(\varepsilon_k = 0) = p, \quad
\prob(\varepsilon_k = 1-\beta) = q,
\]
and where $q = 1-p$. We restrict to the parameters $0 < \beta \leq 1/2$ and $0 < p < 1$ but exclude the case $\beta=p=1/2$. We first determine a distribution function $F_{\beta,p}$ for which the variables $X_k$ are identically distributed. These distributions are singular, and the Cantor function is obtained in the particular case $\beta=1/3$ and $p=1/2$. It will be shown that $F_{\beta,p}$ lies in the domain of attraction of a max-semistable distribution.  The main result is to prove an extreme value law for the process \eqref{eq:AR_proc}. In fact, this law can be obtained by including an extremal index in the law for an i.i.d.\ process with the same marginal distribution. We also point out the connection with deterministic dynamical systems and the relevance of max-semistable distributions in that context.


\newpage

\section{The marginal distribution}

Let $F_{k}$ denote the distribution function of the variable $X_k$ in \eqref{eq:AR_proc}. Conditioning on the variable $\varepsilon_{k+1}$ gives the following recursion:
\begin{equation}
\label{eq:Fk_recursion}
F_{k+1}(x) = p F_{k}(x/\beta) + q F_{k}(x/\beta+1-1/\beta).
\end{equation}
Therefore, the distribution function for which the variables $X_k$ are identically distributed must satisfy the following functional equation:
\begin{equation}
\label{eq:functional}
F(x) = p F(x/\beta) + q F(x/\beta+1-1/\beta).
\end{equation}
The aim of this section is to show that there exists a unique distribution function that solves this equation. To that end, we first derive three properties which are implied by satisfying the functional equation.

\begin{lemma}
\label{lemma:symmetry}
If $F$ satisfies the functional equation \eqref{eq:functional}, then $G(x) = 1-F(1-x)$ satisfies the same equation in which the roles of $p$ and $q$ are interchanged.
\end{lemma}

\begin{proof}
We have that
\[
\begin{split}
G(x)
	& = 1-F(1-x) \\
	& = 1 - p F((1-x)/\beta) - q F((1-x)/\beta+1-1/\beta) \\
	& = p \big[ 1 - F(1-(x/\beta+1-1/\beta)) \big] + q \big[ 1 - F(1-x/\beta) \big] \\
	& = p G(x/\beta+1-1/\beta) + q G(x/\beta),
\end{split}
\]
which completes the proof.
\end{proof}

\begin{lemma}
\label{lemma:support}
If $F$ is a distribution function that satisfies the functional equation \eqref{eq:functional}, then
\begin{enumerate}[(i)]
\item $F(x) = 1$ for all $x \geq 1$;
\item $F(x) = 0$ for all $x \leq 0$.
\end{enumerate}
\end{lemma}

\begin{proof}
(i) It suffices to show that $F(1)=1$. Note that the sequence $(x_n)$ defined by $x_1 = 1/\beta$ and $x_{n+1} = x_n/\beta+1-1/\beta$ is strictly increasing and unbounded. We claim that $F(x_n) = F(1)$ for all $n \in \mathbb{N}$. We have $F(1)=pF(1/\beta)+qF(1)$, which gives $F(x_1)=F(1)$. If the assertion holds for some $n$, then on the one hand we have
\[
\begin{split}
F(1)
	& = F(x_n) \\
	& = p F(x_n/\beta) + q F(x_n/\beta+1-1/\beta) \\
	& \geq p F(x_n/\beta+1-1/\beta) + q F(x_n/\beta+1-1/\beta) \\
	& = F(x_{n+1}).
\end{split}
\]
On the other hand, $x_n < x_{n+1}$ implies $F(1) = F(x_n) \leq F(x_{n+1})$. Hence, the claim follows by induction. Since $F(x_n) \to 1$ as $n \to \infty$ we have $F(1)=1$.

(ii) The sequence $(y_n)$ given by $y_1 = 1-1/\beta$ and $y_{n+1} = y_n/\beta$ is strictly decreasing and unbounded. Similar reasoning shows that $F(0)  = F(y_n)$ for all $n \in \mathbb{N}$ and hence $F(0) = 0$.
\end{proof}

\begin{lemma}
\label{lemma:unique}
If $F$ and $G$ are distribution functions that both satisfy the functional equation~\eqref{eq:functional}, then $F=G$.
\end{lemma}

\begin{proof}
Lemma \ref{lemma:support} implies that $F(x)=G(x)$ when $x \leq 0$ or $x \geq 1$ and thus
\[
\begin{split}
F(x)-G(x)
	& = \begin{cases}
			p(F(x/\beta)-G(x/\beta))		& \text{if } 0 \leq x \leq \beta, \\
			q(F(x/\beta+1-1/\beta)-G(x/\beta+1-1/\beta))	& \text{if } 1-\beta \leq x \leq 1, \\
			0					& \text{otherwise}.
		\end{cases}
\end{split}
\]
Hence, we have the inequality
\[
\|F-G\|_\infty \leq \max\{p,q\} \|F-G\|_\infty.
\]
Since $\max\{p,q\} < 1$ it follows that $F=G$.
\end{proof}

\begin{proposition}
\label{prop:solution_functional}
There exists a unique distribution function $F_{\beta,p}$ satisfying the functional equation \eqref{eq:functional}, which is given by
\begin{equation}
\label{eq:Fp_formula}
F_{\beta,p}(x)
=
\begin{cases}
0 & \text{if } x \leq 0, \\
x^{\log p / \log\beta}\nu_{\beta,p}(\log x) & \text{if } 0 < x \leq 1, \\
1 & \text{if } x > 1,
\end{cases}
\end{equation}
where $\nu_{\beta,p} : \mathbb{R} \to \mathbb{R}$ is positive, bounded, and has period $|\log \beta|$. In addition, the following properties hold:
\begin{enumerate}[(i)]
\item $F_{\beta,p}(\beta x) = pF_{\beta,p}(x)$ for $0 < x < 1$;
\item $1-F_{\beta,p}(1-x) = F_{\beta,q}(x)$ for all $x \in \mathbb{R}$;
\item $1-F_{\beta,p}$ is not regularly varying at $x=1$.
\end{enumerate}
\end{proposition}

\begin{proof}
Consider the sequence of distribution functions $(F_k)$ given by the recursion \eqref{eq:Fk_recursion} where the initial distribution is given by
\[
F_0(x)
=
\begin{cases}
0 & \text{if } x < 0, \\
x & \text{if } 0 \leq x \leq 1, \\
1 & \text{if } x > 1.
\end{cases}
\]
An induction argument shows that all functions $F_k$ satisfy $F_k(x) = 0$ for $x \leq 0$ and $F_k(x) = 1$ for $x \geq 1$. An argument similar to the one in the proof of Lemma \ref{lemma:unique} gives
\[
\|F_{k+1}-F_{k}\|_\infty \leq \max\{p,q\} \|F_{k}-F_{k-1}\|_\infty.
\]
In particular, the sequence $(F_{k})$ converges uniformly to a distribution function $F_{\beta,p}$. By taking the limit $k\to\infty$ in equation \eqref{eq:Fk_recursion} it follows that $F_{\beta,p}$ satisfies the functional equation~\eqref{eq:functional}. The uniqueness of $F_{\beta,p}$ follows from Lemma \ref{lemma:unique}.

For fixed $-\infty < x \leq 0$ define
\[
\nu_{\beta,p}(x) = \lim_{k\to\infty} (e^{x})^{-\log p/\log\beta} F_{k}(e^x).
\]
Clearly, this defines a positive and bounded function $\nu_{\beta,p} : (-\infty,0] \to \mathbb{R}$. For all integers $k \geq 0$ equation \eqref{eq:Fk_recursion} gives
\[
\begin{split}
F_{k}(e^{x+\log \beta}) 
	& = F_{k}(\beta e^{x}) \\
	& = p F_{k-1}(e^{x}) + q F_{k-1}(e^{x}+1-1/\beta) \\
	& = p F_{k-1}(e^{x}),
\end{split}
\]
where we have used that $F_{k-1}(e^{x}+1-1/\beta)=0$ since $e^{x}+1-1/\beta \leq 0$. This gives
\[
(e^{x+\log \beta})^{-\log p/\log\beta}F_k(e^{x+\log\beta})
	= \frac{(e^{x})^{-\log p/\log\beta}}{p}F_k(e^{x+\log\beta})
	= (e^x)^{-\log p/\log\beta} F_{k-1}(e^{x}).
\]
Taking the limit $k\to\infty$ on both sides gives $\nu_{\beta,p}(x+\log\beta)=\nu_{\beta,p}(x)$. We can extend $\nu_{\beta,p}$ to the whole real line by periodicity.

Property (i) follows directly from equation~\eqref{eq:Fp_formula}.

By construction, the function $F_{\beta,q}$ satisfies the functional equation \eqref{eq:functional} in which the roles of $p$ and $q$ are interchanged, but by Lemma \ref{lemma:symmetry} the same holds for $1-F_{\beta,p}(1-x)$. Therefore, property (ii) follows by Lemma \ref{lemma:unique}.

Finally, property (iii) follows from
\[
\frac{1-F_{\beta,p}(1-xh)}{1-F_{\beta,p}(1-h)}
	= \frac{F_{\beta,q}(xh)}{F_{\beta,q}(h)}
	= x^{\log q/\log\beta} \frac{\nu_{\beta,q}(\log x + \log h)}{\nu_{\beta,q}(\log h)},
\]
which does not have a limit as $h\to 0$ for all $x > 0$.
\end{proof}

Figures \ref{fig:example1}--\ref{fig:example3} show plots of the functions $F_{\beta,p}$ and $\nu_{\beta,p}$ for different choices of the parameters $\beta$ and $p$. The proof of Proposition \ref{prop:solution_functional} also shows that the distribution function $F_{\beta,p}$ is continuous. A more precise quantitative formulation of this fact is given in the following lemma which will be needed in Section \ref{sec:dependent}.

\begin{lemma}
\label{lemma:continuity}
Let $F_{\beta,p}$ given by \eqref{eq:Fp_formula}.
For all integers $k \geq 0$ the following implication holds:
\[
|x-y| \leq \beta^k
\quad\Rightarrow\quad
|F_{\beta,p}(x)-F_{\beta,p}(y)| \leq \max\{p,q\}^k.
\]
\end{lemma}

\begin{proof}
If $y \leq x \leq y + \beta^k$, then
\[
0 \leq F_{\beta,p}(x) - F_{\beta,p}(y) \leq F_{\beta,p}(y+\beta^k) - F_{\beta,p}(y).
\]
For $p \leq q$ we have
\[
F_{\beta,p}(y+\beta^k) - F_{\beta,p}(y) \leq F_{\beta,p}(1) - F_{\beta,p}(1-\beta^k) = F_{\beta,q}(\beta^k) = q^k,
\]
and for $p > q$ we have
\[
F_{\beta,p}(y+\beta^k) - F_{\beta,p}(y) \leq F_{\beta,p}(\beta^k) - F_{\beta,p}(0) = p^k.
\]
This completes the proof.
\end{proof}

The following result shows that $F_{\beta,p}$ is in the domain of attraction of a max-semistable distribution:

\begin{proposition}
With the sequences $a_n = 1/\beta^n$, $b_n=1$, and $k_n = \lfloor 1/q^n\rfloor$, we have for all $x < 0$ that
\begin{equation}
\label{eq:F_MSSDA}
\lim_{n\to\infty} F_{\beta,p}(x/a_n + b_n)^{k_n}
	= \exp\{-(-x)^{\log q/\log\beta}\nu_{\beta,q}(\log(-x))\}.
\end{equation}
\end{proposition}

\begin{proof}
Fix $x < 0$ and let $n \in \mathbb{N}$ be sufficiently large so that $0 < -\beta^n x < 1$. Proposition \ref{prop:solution_functional} gives
\[
\begin{split}
F_{\beta,p}(\beta^n x+1)
	& = 1 - F_{\beta,q}(-\beta^n x) \\
	& = 1 - (-\beta^n x)^{\log q/\log\beta} \nu_{\beta,q}(\log(-\beta^n x)) \\
	& = 1 - q^n (-x)^{\log q/\log\beta} \nu_{\beta,q}(\log(-x)).
\end{split}
\]
Since $k_n \to \infty$ and $k_n q^n \to 1$ the result follows.
\end{proof}

Up to a change of scale and location the right-hand side of \eqref{eq:F_MSSDA} indeed of the form~\eqref{eq:maxsemistable_family}. Domains of attraction for max-semistable laws can also be understood in terms of regularity of the tails. Necessary and sufficient conditions in terms of comparisons with regularly varying distributions are formulated in \cite{Megyesi:02}. A very different characterization of max-semistable laws was derived in \cite{CantoHaanTemido:2002}. The latter approach has the advantage that there is no need to know (or guess) a regularly varying distribution function to compare with and also enables the development of estimators for the log-periodic component \cite{CantoDiasTemido:2009}.

\begin{figure}[h]
\centering
\includegraphics{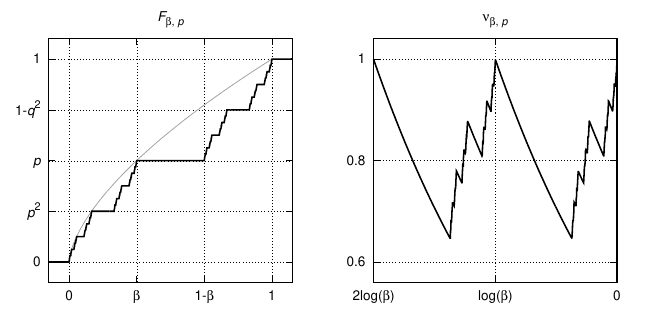}
\caption{Graphs of the distribution function $F_{\beta,p}$, the curve $y=x^{\log p/\log\beta}$ (left panel), and the periodic function $\nu_{\beta,p}$ (right panel). The chosen parameters are $\beta = 1/3$ and $p = 1/2$ in which case $F_{\beta,p}$ is the Cantor function. For other choices of $0 \leq \beta < 1/2$ and $0 < p < 1$ the graphs are qualitatively similar. The parameters $\beta$ and $p$ respectively control the width and height of the plateau's on which $F_{\beta,p}$ is constant.}
\label{fig:example1}
\end{figure}

\begin{figure}[h]
\centering
\includegraphics{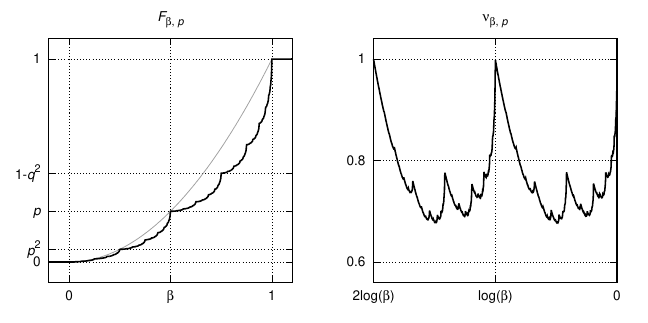}
\caption{As Figure \ref{fig:example1}, but for the parameters $\beta = 1/2$ and $p = 1/4$. For $0 < p < 1/2$ the graphs are qualitatively similar. If $\beta=p=1/2$, then $\nu_{\beta,p} \equiv 1$ and $F_{\beta,p}$ is the uniform distribution on $[0,1]$.}
\label{fig:example2}
\end{figure}

\begin{figure}[h]
\centering
\includegraphics{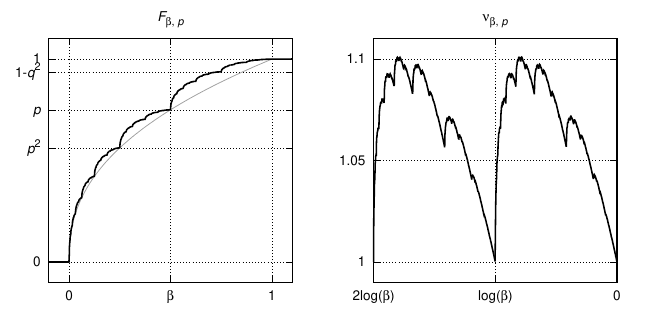}
\caption{As Figure \ref{fig:example1}, but for the parameters $\beta = 1/2$ and $p = 3/4$. For $1/2 < p < 1$ the graphs are qualitatively similar.}
\label{fig:example3}
\end{figure}

\clearpage


\newpage

\section{The extreme value law}
\label{sec:dependent}

In this section we derive a max-semistable extreme value law for the stationary process~\eqref{eq:AR_proc} with the marginal distribution function $F_{\beta,p}$ given by Proposition~\ref{prop:solution_functional}. The main result is the following:
\begin{theorem}
\label{thm:main_result}
With the sequences $a_n = 1/\beta^n$, $b_n=1$, and $k_n = \lfloor 1/q^n\rfloor$ the autoregressive process \eqref{eq:AR_proc} with marginal distribution \eqref{eq:Fp_formula} satisfies the following extreme value law for all $x < 0$:
\begin{equation}
\label{eq:main_result}
\lim_{n\to\infty} \prob(a_n(M_{k_n}-b_n) \leq x) = \exp\{-p(-x)^{\log q/\log\beta}\nu_{\beta,q}(\log(-x))\}.
\end{equation}
\end{theorem}

Comparing the distributions \eqref{eq:F_MSSDA} and \eqref{eq:main_result} shows the appearance of an extremal index which is given by $\theta=p$. The usual interpretation of $\theta$ as the reciprocal of the mean cluster size of extremes implies that for smaller $p$ the extremes on average form larger clusters, which is intuitively clear from \eqref{eq:AR_proc}.

Although the distributions \eqref{eq:F_MSSDA} and \eqref{eq:main_result} share the same tail index $\xi$ and function $\nu$, they are not necessarily of the same type in the sense that the distributions can be obtained from each other by means of an affine rescaling of the variable~$x$. In fact, according to \cite[Corollary of Theorem 8]{TemidoCanto:2003} the distributions \eqref{eq:F_MSSDA} and \eqref{eq:main_result} are of the same type if and only if there exists exists an integer $m>0$ such that $\theta = c^{-m}$, where from \eqref{eq:Megyesi:1.2} we have $c=1/q$. This is the case if and only if $1-q=q^m$ for some integer $m>0$. The latter equation has precisely one solution $q$ for each $m$.

The proof of Theorem \ref{thm:main_result} borrows from \cite{Chernick:1981}, making appropriate adaptations where needed. Throughout this section we make use of the following notation. For a fixed $x < 0$ we consider the following sequence of levels:
\[
u_n = 1 + \beta^n x.
\]
In addition, let
\[
j_n	= \max\{i \in \mathbb{Z} \,:\, u_{n-i} > 0\}.
\]
For simplicity the dependence on $x$ is suppressed in the notation. We only consider values of $n$ sufficiently large so that $0 < u_n < 1$ and $j_n \geq 1$. For all $i=0,\dots,j_n$ we have
\[
F_{\beta,p}(u_{n-i}) = 1-q^{n-i}\psi_{\beta,q}(x)
\quad\text{where}\quad
\psi_{\beta,q}(x) = (-x)^{-\log q/\log \beta}\nu_{\beta,q}(\log(-x)).
\]
Further, observe that $u_{n-i}$ is decreasing in $i$ and $j_n = O(n)$.

Next, consider the following quantities:
\[
P_{s,i}	= \prob(M_s \leq u_n, X_{s-1} \leq u_{n-i}),
\]
for which we have the special cases
\[
P_{1,i} = \prob(X_0 \leq u_{n-i})
\quad\text{and}\quad
P_{s,0} = \prob(M_s \leq u_n).
\]
Again, for simplicity the dependence on both $x$ and $n$ is suppressed in the notation.

\begin{lemma}
\label{lemma:P_recursion}
For all $s \geq 2$ and $0 \leq i \leq j_n-1$ the following recursion holds:
\[
P_{s,i} = p P_{s-1,0} + q P_{s-1,i+1}.
\]
Moreover, for all $2 \leq s \leq j_n+1$ we have:
\[
P_{s,0} = \sum_{i=0}^{s-2} pq^i P_{1,i} + q^{s-1} P_{1,s-1}.
\]
\end{lemma}

\begin{proof}
Since $u_{n-i} \leq u_n$ for all $i \geq 0$ we have
\[
\begin{split}
P_{s,i}
	& = \prob(M_{s-1} \leq u_n, X_{s-1} \leq u_n, X_{s-1} \leq u_{n-i}) \\
	& = \prob(M_{s-1} \leq u_n, X_{s-1} \leq u_{n-i}) \\
	& = \prob(M_{s-1} \leq u_n, \beta X_{s-2} + \varepsilon_{s-1} \leq u_{n-i}).
\end{split}
\]
Conditioning on the variable $\varepsilon_{s-1}$ gives
\[
P_{s,i} = p \prob(M_{s-1} \leq u_n, X_{s-2} \leq u_{n-i}/\beta) + q \prob(M_{s-1} \leq u_n, X_{s-2} \leq u_{n-i}/\beta+1-1/\beta).
\]
Note that $u_{n-i}/\beta+1-1/\beta = u_{n-(i+1)}$. Since $i+1 \leq j_n$ we have $u_{n-i}/\beta > 1$, which implies that the inequality $X_{s-2} \leq u_{n-i}/\beta$ is already implied by $M_{s-1} \leq u_n$ and hence can be omitted in the first term.

The second part of the lemma follows by repeated application of the recursion.
\end{proof}

\begin{lemma}
For all $2 \leq s \leq j_n+1$ we have
\[
\prob(M_{s} \leq u_n) = 1 - (p(s-1)+1)q^n\psi_{\beta,q}(x).
\]
Moreover, we have
\[
\lim_{n\to\infty} \prob(M_{j_n+1} \leq u_n)^{\lfloor k_n / j_n\rfloor} = \exp\{-p\psi_{\beta,q}(x)\}.
\]
\end{lemma}

\begin{proof}
Lemma \ref {lemma:P_recursion} gives
\[
\prob(M_{s} \leq u_n)
	= \sum_{i=0}^{s-2} pq^i F_{\beta,p}(u_{n-i}) + q^{s-1} F_{\beta,p}(u_{n-(s-1)}).
\]
Substituting $F_{\beta,p}(u_{n-i}) = 1-q^{n-i}\psi_{\beta,q}(x)$ and simplifying the expressions leads to the first statement of the lemma.

The second statement follows from the first statement with $s=j_n+1$ and noting that for $k_n = \lfloor 1/q^n\rfloor$ we have $\lfloor k_n/j_n\rfloor \to \infty$ and 
\[
\lim_{n\to\infty} \lfloor k_n/j_n\rfloor (pj_n+1)q^n = p.
\]
This completes the proof.
\end{proof}

The main result \eqref{eq:main_result} follows after proving the following:

\begin{lemma}
\label{lemma:key_lemma}
For fixed $x < 0$ we have
\[
\lim_{n\to\infty} \big|\prob(M_{k_n} \leq u_n) - \prob(M_{j_n+1} \leq u_n)^{\lfloor k_n/j_n\rfloor}\big| = 0.
\]
\end{lemma}

\begin{proof}
Write $r_n = \lfloor k_n/j_n\rfloor$ and divide the integers $\{0,\dots,r_n j_n-1\}$ into intervals as follows:
\[
\begin{split}
I_i		& = \{(i-1)j_n,\dots,i j_n-\ell_n-1\}, \\
I_i^*	& = \{ij_n-\ell_n-1,\dots,i j_n-1\},
\end{split}
\]
where $i=1,\dots,r_n$. The intervals $I_i$ all have length $j_n-\ell_n$ and are separated by gaps $I_i^*$ of length $\ell_n$. We will choose the sequence $(\ell_n)$ such that $\ell_n \to \infty$ and $\ell_n / j_n \to 0$ as $n\to\infty$. Observe that
\begin{equation}
\label{eq:key_limits}
j_n q^n\to 0
\quad\text{and}\quad
r_n \ell_n q^n \to 0.
\end{equation}
For any finite set of nonnegative integers $I$ we write $M(I) = \max\{X_i \,:\, i \in I\}$.

We have the following upper bound:
\[
|\prob(M_{k_n} \leq u_n) - \prob(M_{j_n+1} \leq u_n)^{r_n}| \leq T_1 + T_2 + T_3 + T_4,
\]
where
\[
\begin{split}
T_1 & = |\prob(M_{k_n} \leq u_n) - \prob(M_{j_n r_n} \leq u_n)|, \\[2mm]
T_2 & = |\prob(M_{j_n r_n} \leq u_n) - \prob(\textstyle\bigcap_{i=1}^{r_n} \{M(I_i) \leq u_n\})|, \\[2mm]
T_3 & = |\prob(\textstyle\bigcap_{i=1}^{r_n} \{M(I_i) \leq u_n\})-\textstyle\prod_{i=1}^{r_n}\prob(M(I_i) \leq u_n)|, \\[2mm]
T_4 & = |\textstyle\prod_{i=1}^{r_n}\prob(M(I_i) \leq u_n) - \prob(M_{j_n+1} \leq u_n)^{r_n}|.
\end{split}
\]

\emph{Proof that $T_1, T_2, T_4 \to 0$.} Since $k_n - j_n < j_n r_n \leq k_n$ we have for $T_1$ that
\[
\begin{split}
0 \leq \prob(M_{j_n r_n} \leq u_n) - \prob(M_{k_n} \leq u_n)
	& = \prob(X_i > u_n \text{ for some } i=j_n r_n + 1,\dots,k_n) \\
	& \leq (k_n-j_n r_n) (1-F_{\beta,p}(u_n)) \\
	& \leq j_n (1-F_{\beta,p}(u_n)) \\
	& = j_n q^n \psi_{\beta,q}(x).
\end{split}
\]
For $T_2$ we have
\[
\begin{split}
0 \leq \prob(\textstyle\bigcap_{i=1}^{r_n} \{M(I_i) \leq u_n\}) - \prob(M_{j_n r_n} \leq u_n)
	& = \prob(X_i > u_n \text{ for some } i \in I_1^* \cup \dots \cup I_{r_n}^*) \\
	& \leq r_n \ell_n q^n\psi_{\beta,q}(x).
\end{split}
\]
By stationarity we have
\[
T_4 = |\prob(M(I_1) \leq u_n)^{r_n} - \prob(M_{j_n+1} \leq u_n)^{r_n}|.
\]
For $0 \leq x \leq y \leq 1$ and $r > 0$ we have $0 \leq y^r - x^r \leq r(y-x)$, which gives
\[
\begin{split}
0
	& \leq \prob(M(I_1) \leq u_n)^{r_n} - \prob(M_{j_n+1} \leq u_n)^{r_n} \\
	& \leq r_n \big[ \prob(M(I_1) \leq u_n) - \prob(M_{j_n+1} \leq u_n) \big] \\
	& = r_n \prob(X_i > u_n \text{ for some } i = j_n-\ell_n,\dots,j_n) \\
	& \leq r_n (\ell_n+1)q^n\psi_{\beta,q}(x).
\end{split}
\]
The terms $T_1$, $T_2$, and $T_4$ all tend to zero as $n \to \infty$ by equation \eqref{eq:key_limits}.

\emph{Proof that $T_3 \to 0$.}
Below we will use the following notation. For all integers $k \geq 0$ and $i \geq 1$ we have the decomposition
\begin{equation}
\label{eq:XY_decomposition}
X_{k+i} = \beta^i X_k + Y_i,
\end{equation}
where $Y_i$ is a discrete random variable such that
\[
\prob(Y_i \leq 1-\beta^i) = 1, \quad
\prob(Y_i = 1-\beta^i) = q^i, \quad
\prob(Y_i < 1-\beta^i) = 1-q^i.
\]
The dependence of $Y_i$ on $k$ will be suppressed in the notation since the distribution of $Y_i$ does not depend on $k$.

In Appendix \ref{sec:associated} it is shown that for the events $A_s = \{ M(I_s) \leq u_n \}$ we have the following inequality:
\begin{equation}
\label{eq:strong_inequality}
0 \leq \prob\bigg(\bigcap_{s=1}^{r_n} A_s\bigg) - \prod_{s=1}^{r_n} P(A_s)
\leq
\sum_{s=2}^{r_n} \bigg(\prob\bigg(A_s \mid \bigcap_{i=1}^{s-1} A_i\bigg) - \prob(A_s) \bigg).
\end{equation}
It suffices to show that each term in the right-hand side is $O(q^n)$ since in that case the left-hand side is $O(k_n q^n / j_n) = O(1/j_n) \to 0$.

We first consider the term for $s=2$. By introducing the function
\[
G(v) = \prob( X_{j_n-\ell_n-1} \leq v \mid X_0 \leq u_n,\dots,X_{j_n-\ell_n-1} \leq u_n)
\]
it follows that
\[
P(A_2 \mid A_1)
	= \int_0^{u_n} \prob(X_{j_n} \leq u_n,\dots,X_{2j_n-\ell_n-1} \leq u_n \mid X_{j_n-\ell_n-1} = v)\, dG(v).
\]
With the decomposition \eqref{eq:XY_decomposition} the integrand can be rewritten as
\[
\begin{split}
	& \prob(X_{j_n} \leq u_n,\dots,X_{2j_n-\ell_n-1} \leq u_n \mid X_{j_n-\ell_n-1} = v) \\[2mm]
	& \hspace{2cm} = \prob(Y_{\ell_n+1} \leq u_n-\beta^{\ell_n+1} v,\dots,Y_{j_n} \leq u_n-\beta^{j_n} v).
\end{split}
\]
Partition the integration interval into $j_n-\ell_n+1$ pieces as follows:
\[
[0,u_n] = [0,u_{n-j_n}] \cup (u_{n-j_n},u_{n-(j_n-1)}] \cup \dots \cup (u_{n-(\ell_n+2)},u_{n-(\ell_n+1)}] \cup (u_{n-(\ell_n+1)}, u_n].
\]
Note that $u_n - \beta^i v < 1-\beta^i$ if and only if $u_{n-i} < v$. If $0 \leq v \leq u_{n-j_n}$, then
\[
\prob(Y_{\ell_n+1} \leq u_n-\beta^{\ell_n+1} v,\dots,Y_{j_n} \leq u_n-\beta^{j_n} v) = 1.
\]
If $u_{n-i} < v$ for some $i \in \{\ell_n+1,\dots,j_n\}$, then
\[
\begin{split}
	& \prob(Y_{\ell_n+1} \leq u_n-\beta^{\ell_n+1} v,\dots,Y_{j_n} \leq u_n-\beta^{j_n} v) \\
	& \hspace{1cm} = \prob(Y_{i} \leq u_n-\beta^{i} v, \dots,  Y_{j_n} \leq u_n-\beta^{j_n} v) \\
	& \hspace{1cm} \leq \prob(Y_{i} \leq u_n-\beta^{i} v) \\
	& \hspace{1cm} \leq \prob(Y_{i} < 1-\beta^{i}) \\
	& \hspace{1cm} = 1-q^{i}.
\end{split}
\]
The above observations, together with $G(0) = 0$ and $G(u_n)=1$, give
\[
\begin{split}
P(A_2 \mid A_1)
	& \leq G(u_{n-j_n}) + \sum_{i=\ell_n+2}^{j_n} (1-q^{i})[G(u_{n-(i-1)}) - G(u_{n-i})] \\
	& \hspace{3cm} + (1-q^{\ell_n+1})[1 - G(u_{n-(\ell_n+1)})].
\end{split}
\]
Rewriting the terms in the sum as
\[
\begin{split}
(1-q^{i})[G(u_{n-(i-1)}) - G(u_{n-i})]
	& = [(1-q^{i-1})G(u_{n-(i-1)}) - (1-q^{i})G(u_{n-i})] \\
	& \hspace{2cm} + pq^{i-1}G(u_{n-(i-1)}),
\end{split}
\]
gives a telescoping expression between the brackets in the right-hand side. Hence, the inequality can be simplified as follows:
\begin{equation}
\label{eq:simplified}
P(A_2 \mid A_1)
	\leq 1-q^{\ell_n+1} + \sum_{i=\ell_n+1}^{j_n-1} pq^{i}G(u_{n-i}) + q^{j_n} G(u_{n-j_n}).
\end{equation}

For $n$ sufficiently large it follows from $X_0=0$ that $X_i \leq u_n$ for all $i=1,\dots,j_n-\ell_n-1$, and hence
\[
\begin{split}
G(v)
	& = \prob(X_{j_n-\ell_n-1} \leq v \mid X_0 \leq u_n, X_1\leq u_n, \dots,X_{j_n-\ell_n-1} \leq u_n) \\
	& \leq \prob(X_{j_n-\ell_n-1} \leq v \mid X_0 = 0, X_1\leq u_n, \dots, X_{j_n-\ell_n-1} \leq u_n) \\
	& = \prob(Y_{j_n-\ell_n-1} \leq v).
\end{split}
\]
Since $Y_{j_n-\ell_n-1} \geq X_{j_n-\ell_n-1}-\beta^{j_n-\ell_n-1}$ with probability one, we find
\[
G(v) \leq \prob(X_{j_n-\ell_n-1}-\beta^{j_n-\ell_n-1} \leq v) = F_{\beta,p}(v + \beta^{j_n-\ell_n-1}).
\]

For $i \leq j_n-1$ the functional equation \eqref{eq:functional} gives
\[
\begin{split}
	& F_{\beta,p}(u_{n-i} + \beta^{j_n-\ell_n-1}) - F_{\beta,p}(u_{n-i}) \\
	& \hspace{1cm} = p\big[ F_{\beta,p}(u_{n-i}/\beta + \beta^{j_n-\ell_n-2}) - F_{\beta,p}(u_{n-i}/\beta)\big] \\
	& \hspace{2cm} + q\big[ F_{\beta,p}(u_{n-i}/\beta+1-1/\beta + \beta^{j_n-\ell_n-2}) - F_{\beta,p}(u_{n-i}/\beta+1-1/\beta)\big] \\
	& \hspace{1cm} = q\big[ F_{\beta,p}(u_{n-(i+1)} + \beta^{j_n-\ell_n-2}) - F_{\beta,p}(u_{n-(i+1)})\big],
\end{split}
\]
where the last equality follows from $u_{n-i}/\beta+1-1/\beta=u_{n-(i+1)}$ and $u_{n-i}/\beta > 1$ since $i \leq j_n-1$.
Write $\tau = \max\{p,q\}$. Repeating the last inequality $j_n-i$ times and using Lemma \ref{lemma:continuity} gives
\[
\begin{split}
F_{\beta,p}(u_{n-i} + \beta^{j_n-\ell_n-1}) - F_{\beta,p}(u_{n-i})
	& = q^{j_n-i} \big[ F_{\beta,p}(u_{n-j_n} + \beta^{i-\ell_n-1}) - F_{\beta,p}(u_{n-j_n}) \big] \\
	& \leq q^{j_n-i} \tau^{i-\ell_n-1}.
\end{split}
\]
Using these upper bounds in \eqref{eq:simplified} gives
\[
\begin{split}
\prob(A_2 \mid A_1)
	& \leq 1-q^{\ell_n+1} + \sum_{i=\ell_n+1}^{j-1} pq^{i}\big[F_{\beta,p}(u_{n-i}) + q^{j_n-i} \tau^{i-\ell_n-1}\big] + q^{j_n} \big[F_{\beta,p}(u_{n-j_n}) + \tau^{j_n-\ell_n-1}\big] \\
	& \leq 1-q^{\ell_n+1} + \sum_{i=\ell_n+1}^{j_n-1} pq^{i}F_{\beta,p}(u_{n-i}) + q^{j_n} F_{\beta,p}(u_{n-j_n}) + \bigg(\frac{p}{1-\tau}+1\bigg)q^{j_n}.
\end{split}
\]
Recalling that $F_{\beta,p}(u_{n-i}) = 1-q^{n-i}\psi_{\beta,q}(x)$ gives
\[
\begin{split}
	& 1-q^{\ell_n+1} + \sum_{i=\ell_n+1}^{j_n-1} pq^{i}F_{\beta,p}(u_{n-i}) + q^{j_n} F_{\beta,p}(u_{n-j_n}) \\
	& \hspace{2cm} = 1 - (p(j_n-\ell_n-1)+1)q^n\psi_{\beta,q}(x) \\
	& \hspace{2cm} \leq 1 - (p(j_n-\ell_n-2)+1)q^n\psi_{\beta,q}(x)  = P(A_2).
\end{split}
\]
We conclude that
\[
0 \leq \prob(A_2 \mid A_1) - \prob(A_2) \leq \bigg(\frac{p}{1-\tau}+1\bigg)q^{j_n}.
\]
Since $j_n=O(n)$ it follows that the right-hand side is $O(q^n)$, as desired.

Finally, for $s \geq 3$ consider
\[
\prob(A_s \mid \textstyle\bigcap_{i=1}^{s-1} A_i) = P(A_s \mid A_{s-1} \cap B_{s-2})
\quad\text{with}\quad
B_{s-2} = \textstyle\bigcap_{i=1}^{s-2} A_i.
\]
With the function
\[
G(v) = \prob(X_{(s-1)j_n-\ell_n-1} \leq v \mid A_{s-1} \cap B_{s-2})
\]
we have
\[
\prob(A_s \mid A_{s-1} \cap B_{s-2})
=
\int_0^{u_n} \prob(X_{(s-1)j_n} \leq u_n,\dots,X_{sj_n-\ell_n-1} \mid X_{(s-1)j_n-\ell_n-1} = v)\,dG(v).
\]
By stationarity the integrand can be rewritten as
\[
\begin{split}
	& \prob(X_{(s-1)j_n} \leq u_n,\dots,X_{sj_n-\ell_n-1} \mid X_{(s-1)j_n-\ell_n-1} = v) \\
	& \hspace{1cm} = \prob(X_{j_n} \leq u_n,\dots,X_{2j_n-\ell_n-1} \leq u_n \mid X_{j_n-\ell_n-1} = v).
\end{split}
\]
Similar to the case $s=2$ it follows that for $n$ sufficiently large we have
\[
\begin{split}
G(v)
	& = \prob(X_{(s-1)j_n-\ell_n-1} \leq v \mid \{X_{(s-2)j_n-\ell_n-1} \leq u_n,\dots,X_{(s-1)j_n-\ell_n-1} \leq u_n\} \cap B_{s-2}) \\
	& \leq \prob(X_{(s-1)j_n-\ell_n-1} \leq v \mid \{X_{(s-2)j_n-\ell_n-1} = 0,\dots,X_{(s-1)j_n-\ell_n-1} \leq u_n\} \cap B_{s-2}) \\
	& = \prob(Y_{j_n-\ell_n-1} \leq v).
\end{split}
\]
By stationarity we have $\prob(A_s) = \prob(A_2)$, and thus for $s \geq 3$ all the terms in the right-hand side of equation \eqref{eq:strong_inequality} have the same upper bound as in the case $s=2$.
\end{proof}


\newpage

\section{Connections with deterministic dynamics}

In this section we show how the autoregressive process \eqref{eq:AR_proc} is related to a deterministic dynamical system and we point out the relevance of max-semistable distributions in the latter context. To that end, consider the map
\begin{equation}
\label{eq:map_f}
f_\beta : [0,1] \to \mathbb{R}, \quad f(x) = \begin{cases} x/\beta & \text{if } 0 \leq x < 1-\beta, \\ x/\beta+1-1/\beta & \text{if } 1-\beta \leq x \leq 1, \end{cases}
\end{equation}
and the set
\[
\Lambda_\beta = \bigcap_{n=0}^\infty f_\beta^{-n}([0,1]).
\]
By construction, $\Lambda_\beta$ is the middle-$(1-2\beta)$ Cantor set. Since $f_\beta$ maps $\Lambda_\beta$ into itself we can view iterations of $f_\beta$ as a discrete-time deterministic dynamical system on the state space $\Lambda_\beta$.

We briefly sketch how every distribution function $F_{\beta,p}$ yields an $f_\beta$-invariant measure on the set $\Lambda_\beta$; see~\cite{Berger:2001, DajaniKalle:2021} for details. The sets
\begin{equation}
\label{eq:generating_sets}
(a,b] \cap \Lambda_\beta
\quad\text{where}\quad
0 \leq a \leq b \leq 1,
\end{equation}
form a semi-algebra. The map defined by
\[
\mu_{\beta,p}((a,b] \cap \Lambda_\beta) = F_{\beta,p}(b) - F_{\beta,p}(a)
\]
can be shown to be countably additive and hence can be extended to a measure on the $\sigma$-algebra generated by the sets in equation \eqref{eq:generating_sets}.

\begin{lemma}
For all measurable sets $A \subseteq \Lambda_\beta$ we have
\[
\mu_{\beta,p}(f_\beta^{-1}(A)) = \mu_{\beta,p}(A).
\]
\end{lemma}

\begin{proof}
It suffices to check this for the generators of the $\sigma$-algebra. Let $0 \leq a \leq b \leq 1$, and observe that
\[
f_\beta^{-1}((a,b]\cap\Lambda_\beta) = ((\beta a, \beta b] \cap \Lambda_\beta) \cup ((\beta a + 1-\beta, \beta b + 1-\beta] \cap \Lambda_\beta).
\]
Using the functional equation \eqref{eq:functional} and $1-1/\beta \leq -1$ gives
\[
\begin{split}
\mu_{\beta,p}((\beta a, \beta b] \cap \Lambda_\beta)
	& = F_{\beta,p}(\beta b) - F_{\beta,p}(\beta a) \\
	& = p F_{\beta,p}(b) + q F_{\beta,p}(b+1-1/\beta) \\
	& \hspace{2cm} - p F_{\beta,p}(a) - q F_{\beta,p}(a+1-1/\beta) \\
	& = p \big [F_{\beta,p}(b) - F_{\beta,p}(a)\big].
\end{split}
\]
Similar reasoning gives
\[
\mu_{\beta,p}((\beta a + 1-\beta, \beta b + 1-\beta] \cap \Lambda_\beta) = q \big [F_{\beta,p}(b) - F_{\beta,p}(a)\big],
\]
and hence
\[
\mu_{\beta,p}(f_\beta^{-1}((a,b]\cap\Lambda_\beta)) = \mu_p((a,b]\cap\Lambda_\beta).
\]
This completes the proof.
\end{proof}

The dynamical system $f_\beta : \Lambda_\beta \to \Lambda_\beta$ is related to the autoregressive process \eqref{eq:AR_proc} as follows. If $X_{k+1} = \beta X_k + \varepsilon_{k+1}$, then considering the two cases $\varepsilon_{k+1}=0$ and $\varepsilon_{k+1}=1-\beta$ results into the relation
\[
X_k = f_\beta(X_{k+1}),
\]
which means that we can interpret the map $f_\beta$ as a right inverse of the process.
This relation, together with the fact that $f_\beta$ preserves the measures $\mu_{\beta,p}$, also shows that the process \eqref{eq:AR_proc} is stationary. Indeed, for any choice of integers $0 \leq i_1 < \dots < i_r$ and $m>0$ we have
\[
\begin{split}
\prob\bigg( \bigcap_{k=1}^r \{X_{i_k} \in [0,u_k]\}\bigg)
	& = \prob\bigg( \bigcap_{k=1}^r \{f_\beta^m(X_{i_k+m}) \in [0,u_k]\}\bigg) \\
	& = \prob\bigg( \bigcap_{k=1}^r \{X_{i_k+m} \in f_\beta^{-m}([0,u_k])\}\bigg) \\
	& = \prob\bigg( \bigcap_{k=1}^r \{X_{i_k+m} \in [0,u_k]\}\bigg).
\end{split}
\]

The map $f_\beta$ exhibits sensitive dependence on initial conditions as its Lyapunov exponent (computed with respect to any of the measures $\mu_{\beta,p}$) is $\lambda = -{\log\beta} > 0$. Treating the initial condition as a random variable $Y_0$ gives a stochastic process of the form $Y_{k+1} = f_\beta(Y_k)$. If $Y_0 \sim F_{\beta,p}$, the process $(Y_k)$ is again stationary with marginal distribution function $F_{\beta,p}$. For a chosen initial condition $y_0 \in \Lambda_\beta$ which is not a periodic point of $f_\beta$ (these form a set of measure zero) the orbit $f_\beta^k(y_0)$ can be seen as a realization of the process $(Y_k)$.

Observe the qualitative differences between the cases $0 < \beta < 1/2$ and $\beta=1/2$. For $\beta=p=1/2$ the autoregressive process \eqref{eq:AR_proc} has uniform marginals and belongs to a family of autoregressive processes studied in \cite{Chernick:1981}. In this case, the map $f_\beta$ of equation~\eqref{eq:map_f} reduces to the well-known doubling map for which extreme value laws were proven along completely different lines by means of self-similarity arguments \cite{Haiman:2018} and Fibonacci numbers \cite{BS:2021}. If $\beta=1/2$ but $p \neq 1/2$ then the distribution $F_{\beta,p}$ is strictly increasing and supported on the interval $[0,1]$ rather than the Cantor set $\Lambda_\beta$.

The extension of extreme value theory to the setting of chaotic deterministic dynamical systems is a rapidly expanding area of research; see \cite{Lucarini2016} and references therein. Given a dynamical system the goal is to find limit laws for large values attained by an observable function evaluated along typical orbits of the system. In this context, the process \eqref{eq:AR_proc} was studied numerically for $\beta=1/3$ in the form of an iterated function system in \cite{LFTV:12}, but in that paper the marginal distribution is in the domain of attraction of a max-stable distribution due to the particular choice of the observable.

In general, extreme value laws for dynamical systems depend on the regularity of both the invariant measure and the observable. Chaotic systems often have invariant sets or attractors with a fractal geometric structure. Indeed, the map $f_\beta$ defined above is an example of this. As a consequence, invariant measures supported on these sets may not have regularly varying tails. This may also explain why for certain types of observables estimates of the tail index exhibit persistent oscillatory behaviour \cite{HVRSB:12, S:2016}. Hence, in addition to max-stable distributions, max-semistable distributions may have a natural place in studying extremal behaviour of these systems \cite{HS:2021}.


\appendix

\newpage

\section{Associated random variables}
\label{sec:associated}

We have the following general result:

\begin{lemma}
\label{lemma:cond_ineq}
For any events $A_1,\dots,A_r$ with $r \geq 2$ we have
\[
\bigg|\prob\bigg(\bigcap_{i=1}^{r} A_i\bigg) - \prod_{i=1}^{r} \prob(A_i)\bigg|
\leq
\sum_{i=2}^{r} \bigg|\prob\bigg(A_i \mid \bigcap_{j=1}^{i-1} A_i\bigg) - \prob(A_i) \bigg|.
\]
\end{lemma}

\begin{proof}
The case $r=2$ is trivial. Now let $r \geq 3$, and observe that for $i=1,\dots,r-2$ we have
\[
\bigg[\prod_{j=r-i+1}^{r} \prob(A_j) \bigg] \times \bigg[{-\prob}\bigg(\bigcap_{j=1}^{r-i} A_j\bigg) + \prob\bigg(A_{r-i} \mid \bigcap_{j=1}^{r-i-1} A_j\bigg)\prob\bigg(\bigcap_{j=1}^{r-i-1} A_j\bigg)\bigg] = 0.
\]
Adding these terms between the difference of
\[
\prob\bigg(\bigcap_{i=1}^{r} A_i\bigg)
	= \prob\bigg(A_r \mid \bigcap_{i=1}^{r-1} A_i\bigg)\prob\bigg(\bigcap_{i=1}^{r-1} A_i\bigg)
\quad\text{and}\quad
\prod_{i=1}^{r} P(A_i)
\]
results in a sum consisting of $r-1$ differences. After taking absolute values the claim follows.
\end{proof}

In the proof of Lemma \ref{lemma:key_lemma} we need a stronger result without absolute bars. This is not possible in general, but for the autoregressive process \eqref{eq:AR_proc} this can be achieved by means of the concept of associated random variables \cite{EsaryProschanWalkup:1967}.

\begin{definition}
A finite set of random variables $X_1,\dots,X_n$ is called \emph{associated} if
\[
\cov(f(X_1,\dots,X_n), g(X_1,\dots,X_n)) \geq 0
\]
for all functions $f$ and $g$ which are nondecreasing in each of their arguments.
\end{definition}

The following properties are derived in \cite{EsaryProschanWalkup:1967}:
\begin{enumerate}[(i)]
\item Subsets of associated variables are again associated.
\item If two sets of associated random variables are independent of one another, then their union is again a set of associated random variables.
\item The set consisting of a single random variable is associated.
\item Nondecreasing functions of associated random variables are associated.
\end{enumerate}

In addition, the following theorem is proven \cite[Theorem 5.1]{EsaryProschanWalkup:1967}:

\begin{theorem}
\label{thm:EPW_thm5.1}
Assume that the variables $X_1,\dots,X_n$ are associated and that the functions $f_1,\dots,f_k$ are nondecreasing in each of their arguments. Then for the variables $Y_i = f_i(X_1,\dots,X_n)$ we have the following inequality:
\[
\prob\bigg( \bigcap_{i=1}^k \{Y_i \leq y_i\}\bigg)
\geq \prod_{i=1}^k \prob(Y_i \leq y_i).
\]
\end{theorem}

For the autoregressive process \eqref{eq:AR_proc} observe that
\[
X_j = \beta^j X_0 + \beta^{j-1}\varepsilon_1 + \dots + \beta\varepsilon_{j-1} + \varepsilon_j
\]
is a nondecreasing function of the variables $X_0,\varepsilon_1,\dots,\varepsilon_j$ which are independent and thus associated. Hence, every finite set of the variables $X_j$ is associated.

Finally, for sets of nonnegative integers $I_1,\dots,I_r$ write $M(I_i) = \max\{X_j \,:\, j \in I_i\}$. Theorem~\ref{thm:EPW_thm5.1} gives
\[
\begin{split}
\prob\bigg( \bigcap_{s=1}^r \{M(I_s) \leq u_n\} \bigg)
	& \geq \prod_{s=1}^{r} \prob(M(I_s) \leq u_n), \\
\prob\bigg( \{M(I_s) \leq u_n\} \mid \bigcap_{i=1}^{s-1} \{M(I_i) \leq u_n\} \bigg)
	& \geq \prob(M(I_s) \leq u_n).
\end{split}
\]
Therefore, in this case it is permitted to omit the absolute value bars in Lemma \ref{lemma:cond_ineq}.


\end{document}